\documentclass{article}
\usepackage{graphicx}

\usepackage[fleqn]{amsmath}
\usepackage{amsthm}
\usepackage{amsmath}
\usepackage{amssymb}
\usepackage{amsbsy}
\usepackage{amsfonts}
\usepackage{mathrsfs}
\usepackage[all]{xy}
\usepackage{amstext}
\usepackage{amscd}
\usepackage[dvips]{epsfig}
\usepackage{psfrag}
\usepackage{enumerate}
\usepackage{flafter}
\allowdisplaybreaks

\theoremstyle{plain}
\newtheorem{thm}{Theorem}[section]
\newtheorem{prop}[thm]{Proposition}
\newtheorem{cor}[thm]{Corollary}
\newtheorem{lem}[thm]{Lemma}
\theoremstyle{definition}
\newtheorem{exa}[thm]{Example}

\newtheorem{rem}[thm]{Remark}

\def\Ker{\mathop{\mathrm{Ker}}\nolimits}

\def\piL{\mathop{\pi_1(S^3 \setminus L)}\nolimits}

\newcommand{\lra}{\longrightarrow}
\newcommand{\ra}{\rightarrow}
\newcommand{\Q}{{\Bbb Q}}
\newcommand{\R}{{\Bbb R}}
\newcommand{\Z}{{\Bbb Z}}
\newcommand{\N}{{\Bbb N}}

\newcommand{\A}{{\cal A }}

\newcommand{\pc}[2]{\mbox{$\begin{array}{c}
\includegraphics[scale=#2]{#1.eps}
\end{array}$}}

\begin{document}
\large
\begin{center}{\bf\Large Milnor-Orr invariants from the Kontsevich invariant}.
\end{center}
\vskip 1.5pc
\begin{center}{Takefumi Nosaka\footnote{
E-mail address: {\tt nosaka@math.titech.ac.jp}
}}\end{center}
\vskip 1pc\begin{abstract}\baselineskip=12pt \noindent

As nilpotent studies in knot theory, we focus on invariants of Milnor, Orr, and Kontsevich.
We show that the Orr invariant of degree $ k $ is equivalent to the tree reduction of the Kontsevich invariant of degree $< 2k $.
Furthermore, we will see a close relation between the Orr invariant and the Milnor invariant, and discuss a method of computing these invariants.

\end{abstract}

\begin{center}
\normalsize
\baselineskip=17pt
{\bf Keywords} \\
\ \ \ Knot, Milnor invariant, nilpotent group, Magnus expansion \ \ \
\end{center}

\large
\baselineskip=16pt
\section{Introduction}
In \cite{Mil1,Mil2}, Milnor defines his $\bar{\mu}$-invariants of links, which extract numerical information from the lower central series of the link groups and the link longitudes. The $\bar{\mu}$-invariants have been studied from topological viewpoints (see \cite{IO,CDM,MY} and references therein, and \cite{KN} for a powerful computation). 
Furthermore, as a homotopical approach to $\bar{\mu}$-invariants, Orr \cite{Orr} introduced an invariant of ``based links". This Orr invariant provides obstruction of slicing links in a nilpotent sense \cite{Orr,IO}. However, since the invariant is defined as a homotopy 3-class of a homotopy group, there are few examples of the computation, and it is not clear whether the invariant is properly a generalization of the $\bar{\mu}$-invariants or not

Meanwhile, in quantum topology, a standard way to study nilpotent information is to carefully observe the tree parts of the Kontsevich invariants or LMO functor; see, e.g., \cite{HM,GL,Mas2}. For example, the first non-vanishing term of $\bar{\mu}$-invariants is equal to that of a tree reduction of the Kontsevich invariant \cite{HM}, with a relation to the Chen integral. Concerning the mapping class group, the Johnson-Morita homomorphism and Goldmann Lie algebra can be also nilpotently studied from the LMO functor (see \cite{Mas,Mas2}). Furthermore, such observations of tree parts sometimes approach the fundamental homology 3-class of 3-manifolds, together with relations to Massey products \cite{GL}.

In this paper, inspired by the works of Massuyeau \cite{Mas,Mas2}, we show that the (Milnor-)Orr invariant can be recovered from the Kontsevich invariant. We should note that these invariants are appropriately graded, and that, given $k\in \mathbb{N}$, the Orr invariant of degree $k$ is defined for any based link $L$ whose $\bar{\mu}$-invariants of degree $\leq k$ are zero. The theorem is as follows:
\begin{thm}
[Corollary \ref{thm2244}]\label{thm22442} Given a based link $L$ whose $\bar{\mu}$-invariants of degree $\leq k$ vanish, the Orr invariant of $ L $ is equivalent to the tree-shaped reduction of the Kontsevich invariant of degree $< 2k $.
\end{thm}
\noindent
This theorem is a generalization of the above result \cite{HM} (see Remark \ref{ll}) and gives a topological interpretation of the tree reduction of degree $< 2k $. Moreover, it is natural for us to ask what finite type invariants recover the Orr invariant; as a solution, we suggest a computation of the Orr invariant from HOMFLYPT polynomials (see \S \ref{SHigher}), where this computation is based on \cite{MY}.

Furthermore, we will show the equivalence between the Orr invariant of degree $k$ and Milnor $\bar{\mu}$-invariants of degree $< 2k$; see Theorem \ref{thm24}. Accordingly, while the Orr invariant is a homotopy 3-class, the 3-class turns out to be described by the link longitudes, as is implicitly pointed out in \cite{Mas2,K}; see the figure below as a summary. Our result is analogous to the result \cite{Mas} concerning the mapping class group, which claims an equivalency between ``the Morita homomorphism" of degree $k$ and ``the total Johnson homomorphisms" of degree $< 2k$. Thus, it can be hoped that and the tree parts of the LMO functor can be described in terms of homology 3-classes or algebraic topology.

\begin{center}
\fbox{
\begin{tabular}{c}
\!\!\!\! Milnor $\mu$-invariant \!\!\!\! \\
of degree $<2k$
\end{tabular}} \raisebox{0.95ex}[0pt][0pt]{$\stackrel{\textrm{Theorem \ref{thm24}}}{ =\!\!=\!\!=\!\!=\!\!=\!\!=\!\!=\!\!=\!\!=}$}
\fbox{
\begin{tabular}{c}
\!\!\!\! Orr invariant \!\!\!\!\\
of degree $k$
\end{tabular}}\raisebox{0.95ex}[0pt][0pt]{$\stackrel{\textrm{Corollary \ref{thm2244}}}{ \ = \!\!=\!\!=\!\!=\!\!=\!\!=\!\!=\!\!=\!\!=\!\!=}$}
\fbox{
\begin{tabular}{c}
\!\! \!\!Tree-reduced Kontsevich \!\!\!\! \\
invariant of degree $<2k$
\end{tabular}}

\end{center}

\vskip -1.97pc
\hspace{46.83ex}
$ \leftarrow \!\! \!- \!\! \!- \!\! \!-\!\! \!-\!\!\! - \!\!\! -\!\!\! - \!\!\! - $

\vskip -0.97pc
\hspace{47.85ex}
{\small Recover via}

\vskip -0.47pc
\hspace{47.845ex}
{\small HOMFLYPT \cite{MY}}

This paper is organized as follows. Section 2 reviews Milnor-Orr invariants and states Theorem \ref{thm24}, and Section 3 describes the relation to the Kontsevich invariant. Section 4 gives the proofs. Appendix \ref{SHigher} explains the computation from HOMFLYPT polynomials.

\section{Theorem on Milnor-Orr invariant}
\label{Semb}
The first part of this section reviews the Milnor-Orr invariant, and the second part states Theorem \ref{thm24}.

\subsection{Review of Milnor invariant and Orr invariant}
\label{Semb5}
We will start by reviewing string links. Let $I$ be the interval $[0,1]$, and fix $q \in \N $. A {\it ($q$-component) string link} is a smooth embedding of $q$ disjoint oriented arcs $A_1, \dots, A_q$ in the 3-cube $I^3$, which satisfy the boundary condition $A_i = \{ p_i, q_i\}$, where $p_i =( j /2q , 0,0)$ and $q_i =( j /2q , 0,1) \in I^3 $. We define $SL(q)$ to be the set of string links of $q$-components. Given two string links $T $ and $T'$, we can define another string link, $T \cdot T' \subset [0,1]^2 \times [0,2] \cong I^3$, by connecting $q_i$ in $T$ and $ p_i'$ in $T'$. If $T =\{A_i \}$ is a string link, then an oriented link $L = \{L_1,\dots, L_q\}$ can be defined to be $L_i = A_i \cup a_i$, where $a_i$ is a semi-circle in $S^3 \setminus L $ connecting $p_i$ and $q_i$. We call the link {\it the closure of $T$} and denote it by $\overline{T}$; see Figure \ref{STU2}.

\begin{figure}[b]
\begin{center}
\begin{picture}(50,74)
\put(-156,37){\pc{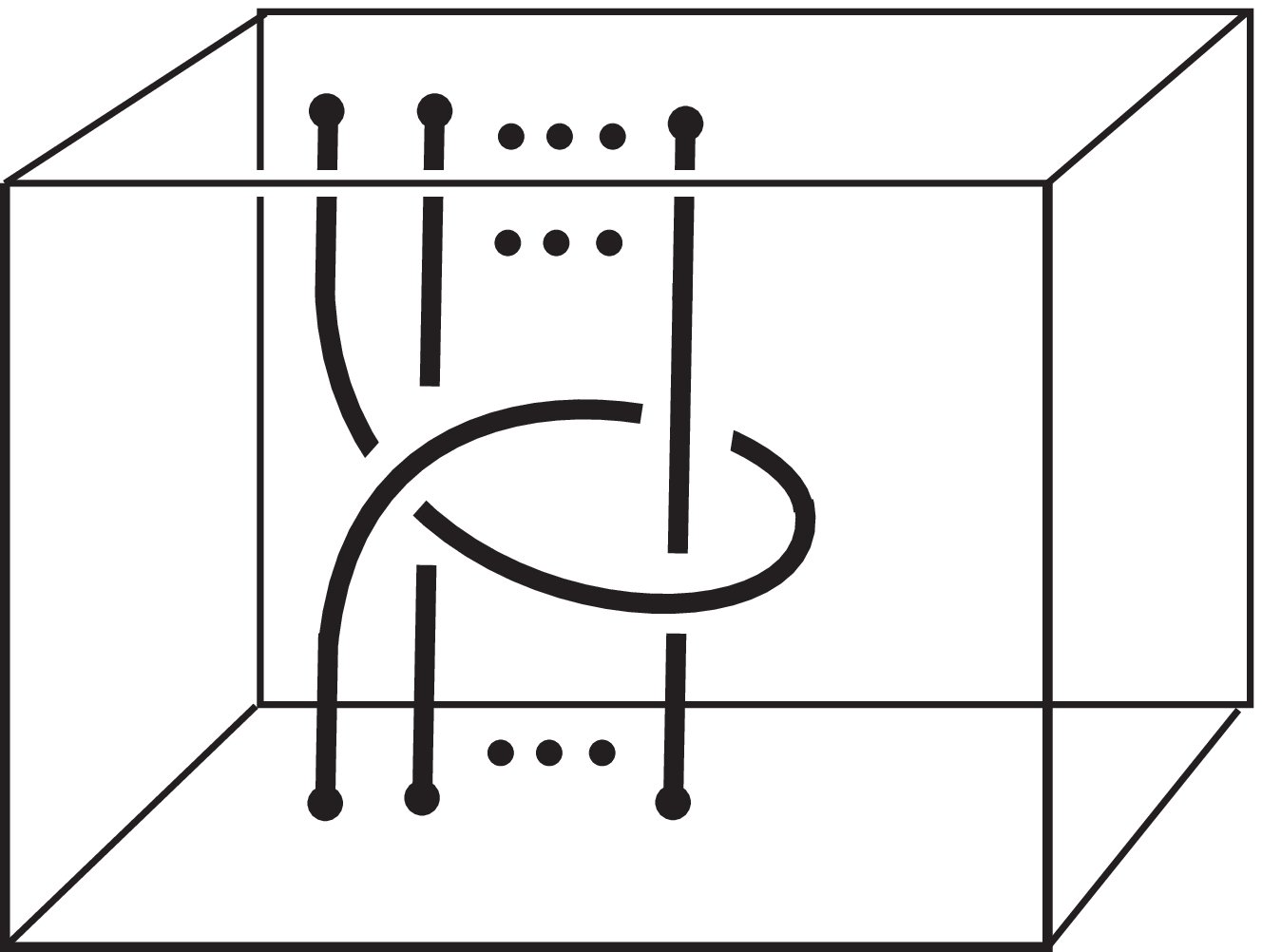}{0.202}}
\put(26,37){\pc{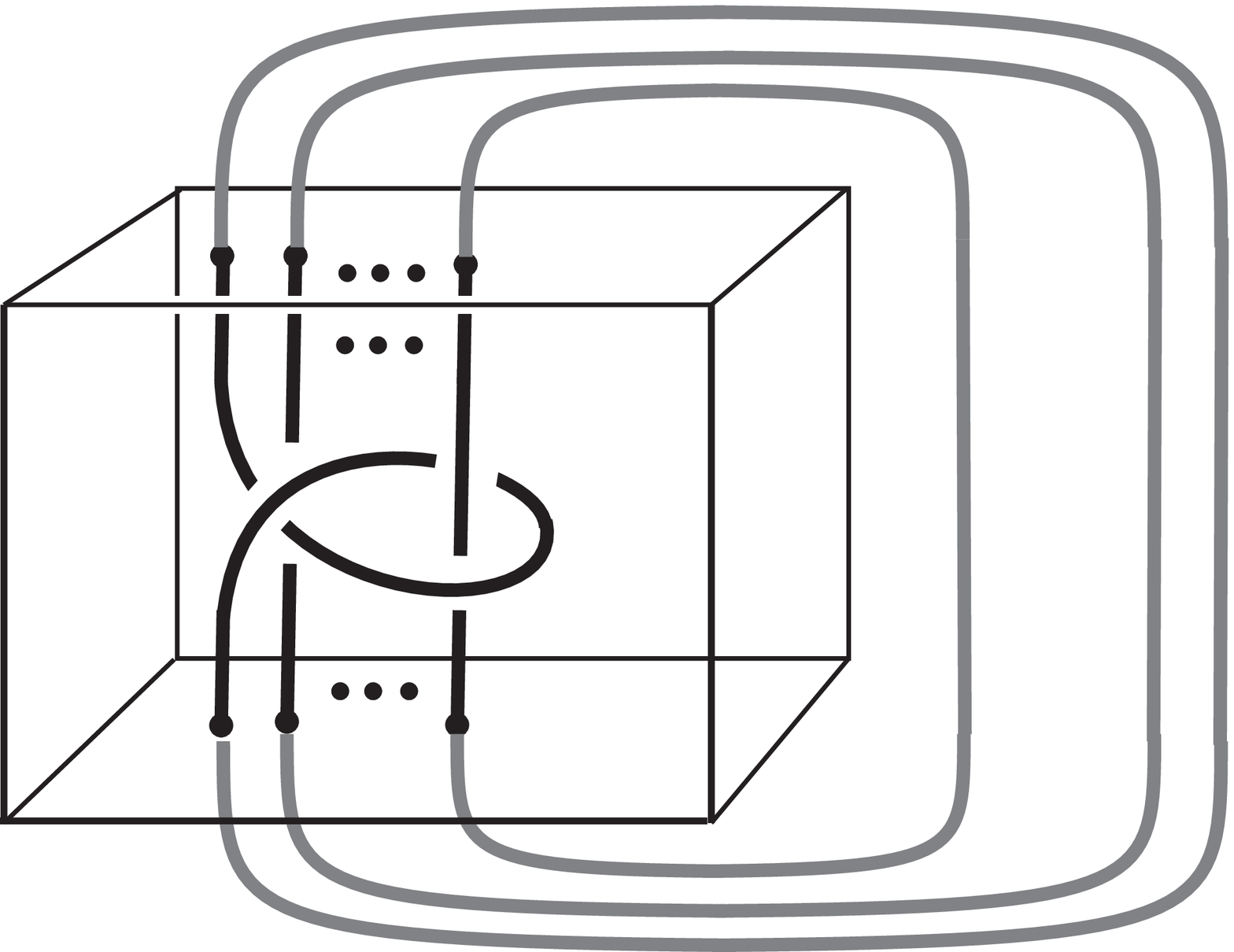}{0.204}}
\end{picture}
\end{center}
\vskip -1.17pc
\caption{\label{STU2} A string link and the closure as a link in $S^3$.}
\end{figure}

Next, we will describe the groups used throughout this paper. Let $F$ be the free group with generator $x_1, \dots, x_q$.
For a group $G$, we define $G_1 $ to be $G$ and $G_m $ to be the commutator $ [G_{m-1} , G]$ by induction. If $G$ is the free group $G$, then the projection $p_{m-1}: F/F_m \ra F/F_{m-1}$ implies the central extension,
\begin{equation}\label{kihon2} 0 \lra F_{m-1} /F_m \lra F/F_m \xrightarrow{\ \ p_{m-1}\ \ }F/F_{m-1} \lra 0 \ \ \ \ \ (\mathrm{central \ extension})
\end{equation}
The abelian kernel $F_{m-1} /F_m $ is known to be free with a finite basis.

Now let us explain the $m$-th leading terms of the Milnor invariant, according to \cite{Mil1,IO,KN}. We will suppose that the reader has elementary knowledge of knot theory, as can be found in \cite[\S 1 and \S 12]{CDM}. Given a $q$-component link $L $ in the 3-sphere $S^3$ and $ \ell \leq q $, we can uniquely define the (preferred) longitude $ \mathfrak{l}_{\ell} \in \pi_1 (S^3 \setminus L)$ of the $\ell$-th component. In addition, let $f_2 : \pi_1 (S^3 \setminus L) \ra F/F_2 =\Z^q $ be the abelianization $\mathrm{Ab}$. Furthermore, for $k \in \mathbb{N}$, we assume
\begin{itemize}
\item
\noindent
\textbf{Assumption $\mathfrak{A}_{k}$}. There are homomorphisms $f_s : \pi_1 (S^3 \setminus L) \ra F/F_s $ for $s$ with $s \leq k$, which satisfy the commutative diagram,
$$
{\normalsize
\xymatrix{
\pi_1 (S^3 \setminus L) \ar[d]_{f_2} \ar[dr]_{f_3} \ar[drrr]_{f_4}^{} \ar[drrrrrr]^{f_k}_{\!\!\!\!\!\!\!\!\!\!\!\!\!\!\!\!\!\!\!\!\!\!\!\!\!\!\!\! \cdots \cdots }& & \\
F/F_2 & F/F_3 \ar[l]^{p_2 }& &F/F_4 \ar[ll]^{p_3} &\cdots \cdots \ar[l] & & F/F_{k} \ar[ll]^{p_{k-1} }.
}}$$
\end{itemize}

\noindent
Here, we should note that if there is another extension $f_k' $ instead of $f_{k}$, then $f_k$ equals $f_k'$ up to conjugacy, by centrality. Further,
we should note following proposition.
\begin{prop}
[\cite{Mil2}]\label{ea211} Suppose Assumption $\mathfrak{A}_{k}$. Then, $f_k$ admits a lift $ f_{k+1} : \pi_1 (S^3 \setminus L) \ra F/F_{k+1} $ if and only if all the Milnor invariants of length $< k $ vanish, i.e., $ f_k ( \mathfrak{l}_{\ell} ) =0 $.
\end{prop}
Thus, the map $f_k$ sends the preferred longitude $ \mathfrak{l}_{\ell} $ to the center $ F_{k-1} /F_k $. Then, the $q$-tuple,
$$ \bigl( f_k( \mathfrak{l}_1 ), \dots, f_k( \mathfrak{l}_{q} ) \bigr) \in (F_{k-1} /F_k)^q, $$
is called {\it the first non-vanishing Milnor $\mu $-invariant} or {\it the Milnor $\mu $-invariant} of length $k-1 $. Proposition \ref{ea211} implies that $\mu $-invariant is known to be a complete obstruction for lifting $f_{m}$.
The paper \cite{KN} gives an algorithm to describe $f_m$ explicitly, and a method of computing the Milnor invariants. 

We further review the Orr invariant \cite{Orr}, where $L$ satisfies Assumption $\mathfrak{A}_{k+1}$, that is, all the $\mu $-invariants of length $\leq k$ are zero; $f_{k } (\mathfrak{l}_{\ell})=0$. Fix a homomorphism $\tau : F \ra \piL $ that sends each generator $x_{\ell}$ to some meridian $ \mathfrak{m}_{\ell}$ of the $\ell$-th component for $\ell \leq \# L$. This $\tau$ is called {\it a basing}, and the pair $(L,\tau)$ is referred as to {\it a based link}. As examples that we will refer to later, given a string link $T$, the closure $\overline{T}$ has a canonical basing, where $\tau$ is obtained from choosing the loop circling $\{( j/2q , 0, 1)\} $. Furthermore, for a group homomorphism $f: G \ra H$, we will write $f_*: K( G,1 ) \ra K( H,1 )$ for the induced map between Eilenberg-MacLane spaces. We define the space $K_k$ to be the mapping cone,
$$K_k := \mathrm{Cone} \bigl( (f_k \circ \tau )_*: K( F,1) \ra K( F/F_k ,1 )\bigr). $$
Then, from the assumption $f_k(\mathfrak{l}_{\ell})=0$, $f_k$ gives rise to a continuous map $\rho_L: S^3 \ra K_k . $ It is reasonable to consider the homotopy 3-class,
\begin{equation}\label{hom3} \theta_k (L, \tau) := [ \rho_L ] \in \pi_3 ( K_k ) ,
\end{equation}
which we call {\it the Orr invariant}. The following is a list of known results on the invariant and $\pi_3 ( K_k ).$

\begin{thm}[\cite{Orr,IO}]\label{thdr}
\begin{enumerate}[(I)]
\item Let $N_h \in \N $ be the rank of $H_2 (F/F_h;\Z ) =F_h/F_{h+1 }$. The following are isomorphisms on $\pi_3(K_k)$ and on $H_3(K_k)$:
$$\pi_3(K_k) \cong \bigoplus_{h=k}^{2k-1} \Z^{q N_{h} - N_{h+1} } , \ \ \ \ \ \ \
H_3(K_k;\Z )\cong H_3(K(F/F_k,1 );\Z) \cong \bigoplus_{h=k}^{2k-2} \Z^{q N_{h} - N_{h+1} } .$$
Furthermore, the Hurewicz homomorphism $\mathfrak{H}: \pi_3(K_k)\ra H_3(K_k;\Z)$ is equal to the projection according to the direct sums on the right-hand sides.
\item The lowest summand of $\theta_k (L, \tau) $ is equivalent to the Milnor invariant of length $k$; see \cite{Orr} or \cite[\S 10]{IO} for details.

\item For any element $ \kappa \in \pi_3(K_k ) $, there exist a link $L$ and a homomorphism $g_k : \pi_1 (S^3 \setminus L) \ra F/F_k$ satisfying $ \theta_k (L, \tau) = \kappa. $

\item The Orr invariant has additivity with respect to ``band connected sums"; see \cite[\S 3]{Orr}. As a special case, for two string links $T_1$ and $T_2$ such that the closures $\overline{T}_1$ and $\overline{T}_2$ satisfy Assumption $\mathfrak{A}_{k+1}$, we have $\theta_k ( \overline{T_1 \cdot T_2}, \tau_1 ) = \theta_k ( \overline{T}_1, \tau_1) + \theta_k ( \overline{T}_2, \tau_2) . $

\item Let $\iota_k: K_k \ra K_{k+1}$ be the continuous map arising from the projection $F/F_k \ra F/F_{k+1}$. The Orr invariant has functoriality. To be precise, if $L$ satisfies $\mathfrak{A}_{k+1}$, then the equality $(\iota_{k})_* \bigl( \theta_k (L, \tau)\bigr) = \theta_{k+1}(L, \tau) $ holds in $\mathrm{Im} (\iota_{k})_* \cap \pi_3(K_{k+1}) $.
\end{enumerate}
\end{thm}

Next, we mention the homological reduction of the Orr invariant $ \theta_k (L)$ via the Hurewicz map $\mathfrak{H}: \pi_3(K_k)\ra H_3(K_k;\Z).$ Note that the inclusion $K(F/F_k,1) \ra K_k $ induces the isomorphism,
$$P^{\rm gr} : H_3(K (F/F_k,1), \sqcup^q K(\Z,1) ;\Z) \cong H_3(K_k;\Z), $$
from the relative homology. To summarize, the value $\mathfrak{H} ( \theta_k (L)$) is the reduction of $ \theta_k (L)$ without the top summand $\Z^{q N_{2k-1} - N_{2k} } $. Moreover, by definition, this $\mathfrak{H} ( \theta_k (L)$) can be regarded as the pushforward of the fundamental 3-class $ [S^3 \setminus L , \partial ( S^3 \setminus L)] \in H_3(S^3 \setminus L , \partial ( S^3 \setminus L) ;\Z )\cong \Z $:
\begin{equation}\label{hom553} P^{\rm gr} \circ (f_k)_* [S^3 \setminus L , \partial ( S^3 \setminus L)] = \mathfrak{H} ( \theta_k (L,\tau)) \in H_3(K_k;\Z ).
\end{equation}
The author \cite{Nos} showed that the cohomology $H^3(K (F/F_k,1)) $ are generated by some Massey products; thus, the reduction \eqref{hom553} are characterized by some Massey products of $S^3 \setminus L $.

\subsection{Results: Orr invariant of higher invariants}\label{revi3}
In order to state the theorem, let us briefly review the Milnor $\mu$-invariant for string links (see \cite{HM,K,Le}). For a string link $T \in SL(q)$, let $y_j \in \pi_1 ( [0,1]^3 \setminus T) $ be an element arising from the loop circling $\{( j/2q , 0, 1)\} $; Let $G_m$ be the $m$-th nilpotent quotient of $\pi_1 ( [0,1]^3 \setminus T) $. Then, as is shown \cite{Mil2,Le}, the homomorphism $ F \ra \pi_1 ( [0,1]^3 \setminus T) $ which sends $ x_i $ to $y_i$ descends to an isomorphism between the $m$-th nilpotent quotients:
\begin{equation}\label{dd} \phi_*: F/F_m \cong \pi_1 ( [0,1]^3 \setminus T)/G_m, \textrm{ \ for any } m \in \N.\end{equation}
Here, we should note that, if $T$ is a pure braid, this $\phi_*$ is the identity map. Furthermore, a framing of the $\ell $-th component of $T$ defines a parallel curve which determines an element, $\lambda_\ell \in \pi_1 ([0,1]^3 \setminus T ) $. This $\lambda_\ell $ is referred as to {\it the $\ell$-th longitude of $T$}. We call the reduction $\phi_*^{-1}(\lambda_\ell) \in F /F_m $ modulo $ F_m $ {\it the $\mu$-invariant of $T$} (of degree $\leq m$). Later, we will omit writing $\phi_*^{-1} $ for simplicity. We should notice, from the definitions, that the closure $\overline{T}$ of a string link $T$ satisfies $\mathfrak{A }_{k+1} $ if and only if $ \lambda_j $ lies in $F_k $, and $\lambda_j = f_{k+1} ( \mathfrak{l}_j ) \in F_k / F_{k +1}$ modulo $ F_{k +1}$.

Thus, for such a string link $T$, as in the paper \cite{Le} of Levine, it is reasonable to consider the invariant, $\lambda_j $ modulo $ F_{2k}$. Notice that $F_k/F_{2k}$ is abelian, since $[F_k,F_k] \subset F_{2k}$. Furthermore, as can be seen from \cite[Proposition 4]{Le}, we can verify that the equality,
\begin{equation}\label{le} [x_1, \lambda_1 ][x_2, \lambda_2] \cdots [x_q, \lambda_q] =1 \in F,
\end{equation}
always holds. Thus, for $m \leq 2k$, the sum,
$$ \sum_{j=1}^q(x_j\otimes \lambda_j ) \in \Z^q \otimes_{\Z } F_k/F_{m} \textrm{ modulo } F_{m},$$
is contained in the kernel of the commuting operator,
$$ [ \bullet , \bullet ]_{k,m} : F/F_2 \otimes F_k/F_{m} \lra F_{k+1} /F_{m+1}; \ \ x\otimes y \longmapsto xy x^{-1} y^{-1}. $$
This operation will be used in many times. Now let us show the equivalence of Milnor and Orr invariants:
\begin{thm}
[{See \S \ref{Sher33} for the proof}]\label{thm24}
(I) There is a $\Q$-vector isomorphism,
$$ \Phi \circ \eta^{-1}: \Q \otimes \Ker ([ \bullet , \bullet ]_{ k, 2k -1 } ) \stackrel{\sim }{\lra } H_3(F/F_k;\Q ), $$
such that the following holds for any string link $T$ satisfying $\mathfrak{A }_{k+1} $ of the closure $\overline{T}$:
$$\Phi \circ \eta^{-1} \bigl( (x_1 \otimes \lambda_1)+ \cdots + (x_q \otimes \lambda_q ) \bigr) = \mathfrak{H} \circ \theta_k (\overline{T} , \tau). $$
(II) Furthermore, concerning the homotopy group $\pi_3 (K_k )$, there is a bijection
$$ \overline{\Phi} : \Q \otimes \Ker ([ \bullet , \bullet ]_{ k, 2k } ) \stackrel{\sim }{\lra } \pi_3(K_k) \otimes \Q, $$ as an extention of $\Phi \circ \eta^{-1}$,
such that a similar equality $ \overline{\Phi} \bigl( (x_1 \otimes\lambda_1)+ \cdots + (x_q \otimes \lambda_q ) \bigr) = \theta_k(\overline{T} , \tau)$ holds for any string link $T$ satisfying $\mathfrak{A }_{k+1} $ of the closure $\overline{T}$.
\end{thm}
We conjecture that the bijection in (II) is an isomorphism. However, from Theorem \ref{thdr} (III), we have the realizability result as a generalization of \cite[Proposition 5]{Le}:
\begin{cor}\label{cor24}
Let $(\alpha_1,\dots, \alpha_q ) \in F_k/F_{2k} $ satisfy $[x_1, \alpha_1 ]\cdots [x_q, \alpha_q] =1 \in F/F_{2k}$. There exists a string link with $\bar{\mu}$-invariants $(\lambda_1, \dots, \lambda_q)\in (F_k)^q $ such that $ \lambda_j \equiv \alpha_j $ modulo $F_{2k}$.
\end{cor}

\section{As a tree reduction of the Kontsevich invariant}
\label{2thm}
As described in the Introduction, we will relate the $\bar{\mu}$-invariants with the Kontsevich invariant.

\

{\bf Notation}. Throughout this paper, the expression $O(n)$ will be used to denote terms of degree greater than or equal to $n$.

\subsection{A brief review of the Kontsevich invariants of (string) links}
\label{review}
Let us start by briefly reviewing the definition of the $\Q$-vector space $\A(\uparrow^q)$, where {\it a chord diagram (of $q$-components)} is a union $ (\sqcup^q_{j=1} [0,1]) \cup \Gamma$ such that $\Gamma$ is a uni-trivalent graph, whose univalent vertices lie in the interior of $ (\sqcup^q_{j=1} [0,1])$, and each component of $\Gamma$ is required to have a univalent vertex. It is customary to refer to the components of $\Gamma$ as dashed. Then, $\A(\uparrow^q)$ is defined by the $\Q$-vector space generated by all chord diagrams subject to the STU, AS, and IHX relations. The three relations are described in Figure \ref{STU}. The space $\A( \uparrow^q )$ is graded by the degree, where the degree of a diagram is half the number of vertices of $\Gamma $. We denote the subspace of $\A(\uparrow^q)$ of degree $n$ by $\A_n(\uparrow^q ) $. By abuse of notation, we will denote the graded completion of $\A(\uparrow^q)$ by $\A(\uparrow^q)$ as well. We define $\A^t(\uparrow^q)$ as the subspace of $\A(\uparrow^q)$ generated by the chord diagram such that all trivalent diagrams containing a non-simply connected dashed component are considered relations. Furthermore, the stucking connection of $(\sqcup^q_{j=1} [0,1])$ and $(\sqcup^q_{j=1} [0,1])$ gives rise to a ring structure of $\A(\uparrow^q)$. Moreover, there is a cocommutative multiplication $\Delta : \A(\uparrow^q) \ra \A(\uparrow^q)^{\otimes 2 } $ (see, e.g., \cite[Chapter 4]{CDM}, for the details), and $\A(\uparrow^q )$ is made into a Hopf algebra.

\begin{figure}[tpb]
\begin{center}
\begin{picture}(50,74)
\put(-174,38){\large \ \ \ \ $=$}
\put(-94,39){\large $- $}

\put(11,38){\large \ \ \ \ $= \ \!\! -$}
\put(128,38){\large \ \ \ \ \ \ $= \ \ \ \ \ \ \ \ \ \ \ \ \ \ -$}

\put(-16,37){\pc{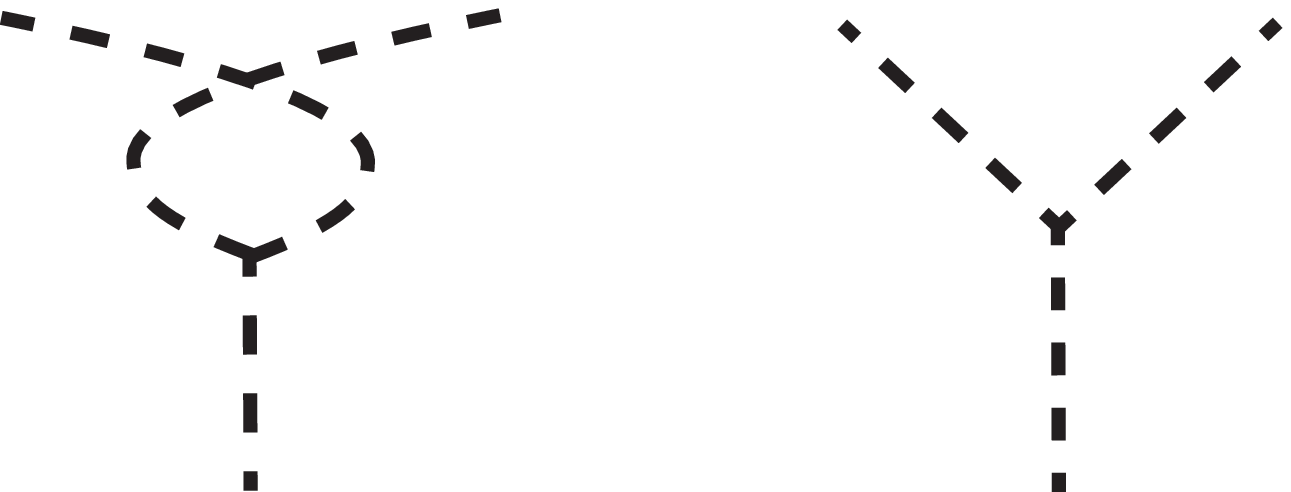}{0.2532}}
\put(-206,37){\pc{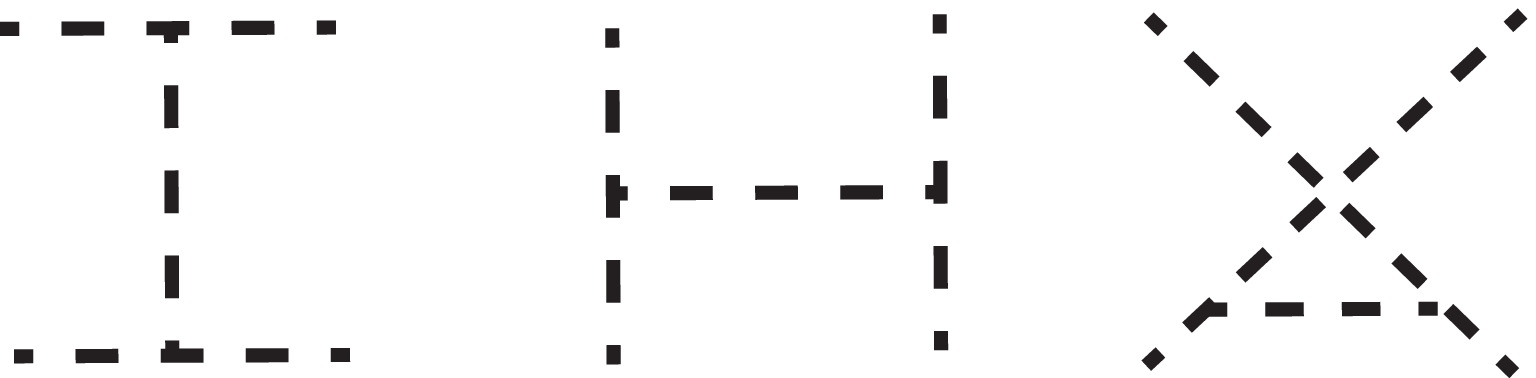}{0.354}}
\put(101,37){\pc{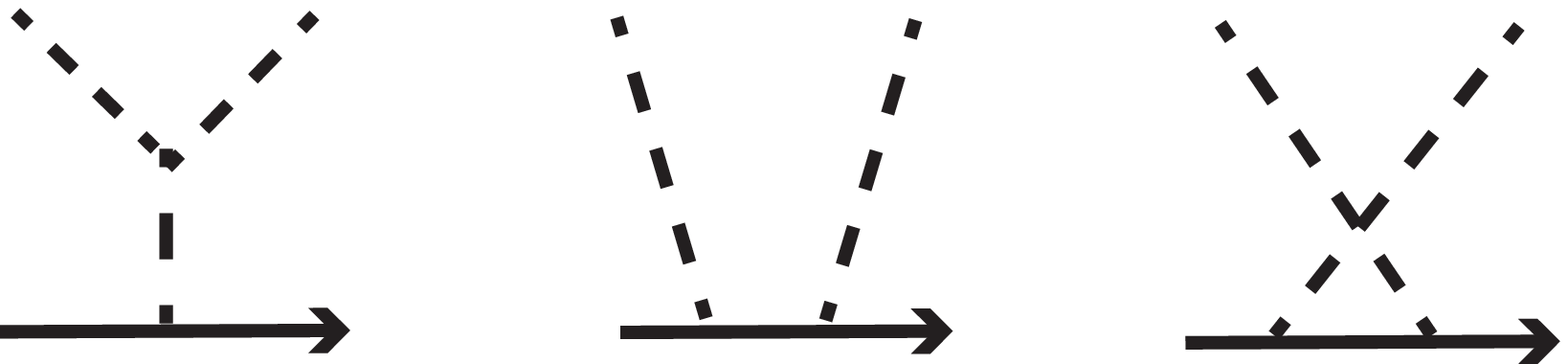}{0.354}}
\end{picture}
\end{center}
\vskip -1.7pc
\caption{\label{STU} The IHX, AS, and STU relations among chord diagrams. }
\end{figure}

There are several formulations of the Kontsevich invariant.
In this paper, we emply 
the Kontsevich invariant defined for ``$q$-tangles", as in \cite{HM}.
Since we can chose an injection from the set of string links into the set of $q$-tangles which is invariant with respect to the composite of ($q$-)tangles,
we can regard the Kontsevich invariant of $q$-tangles that of string links (Throughout this paper, we fix such an injection).
In this paper, we only consider some of the properties of the Kontsevich invariant of string links $T$.
Thus, while we refer the reader to, e.g., \cite[\S\S 8--10]{CDM}, \cite[\S 6]{Mas2}, and references therein for the definition of the Kontsevich invariant for $q$-tangles, here are the properties that we will use later.
\begin{enumerate}[(I)]
\item The invariant, $ Z(T)$, is defined as an element of $ \A (\uparrow^q)$.
\item $Z(T) $ is multiplicative, i.e., $Z(T_1) Z(T_2) = Z (T_1 \cdot T_2)$ holds for two string links $T_1 , T_2$.
\item Every $Z(T)$ is group-like in $ \A (\uparrow^q)$, i.e., $\Delta(Z(T) ) = Z(T) \otimes Z(T)$.
\item (Doubling Formula) For $i \leq q$ and $T \in \A (\uparrow^q)$, one has that
\[ \Delta_i (Z(T))=Z ( D_i(T)) \in \mathcal{A}(\uparrow^{q+1})\]
where $D_i (T) $ denotes the double of $T $ along the $i$-th component, $C_i $, and $\Delta_i $ of a diagram is the sum over all lifts of vertices on $C_i$, to vertices on the two components over $C_i$.
\end{enumerate}


Recalling the tree subspace $\A^t(\uparrow^q)$, we take the projection $p^t: \A (q) \ra \A^t(\uparrow^q)$. We denote by $Z^t_{<m}(T)$ the composite $p^t \circ Z (T)$ subject to $O(m)$. In other words, this $Z^t_{<m}(T)$ is a tree reduction of the Kontsevich invariant $ Z_{<m}(T)$.

Furthermore, we need to use the notion of primitive subspaces. Here, $m \in \mathcal{A}(\uparrow^q)$ is called {\it primitive} if $\Delta(m) = 1 \otimes m + m \otimes 1 $. We denote by $\mathcal{P}^t(\uparrow^q)$ the subspace of primitive elements of $\A^t (\uparrow^q)$, and by $\mathcal{P}^t_{h}(\uparrow^q )$ the subspace of degree $h$. As is known, $\mathcal{P}^t(q )$ is the graded subspace of $\A^t (\uparrow^q)$ generated by chord diagrams such that the the dashed graph $\Gamma $ is simply connected.
Furthermore, the rank of $\mathcal{P}^t_{h}(\uparrow^q )$ is known to be $q N_h-N_{h+1}$; see \eqref{bbcc} and Theorem \ref{thma1}.

\subsection{Results}
\label{review5}
Before stating the theorem, we should mention the following easily proven lemma.
\begin{lem}\label{lem122}
(1) Fix $k \in \Z$. Every elements $a , b $ in $\oplus_{h=k}^{2k-1} \A_h (\uparrow^q) $ satisfy $(1+a) \cdot (1+b )\equiv 1+a + b+ O (2k) \in \A (\uparrow^q)$.

\noindent
(2) \ Let $a \in \A (\uparrow^q) $ satisfy $ a = 1+ O(k )$, and $\Delta (a) = a \otimes a$. If we decompose $a =1 +b +c $ such that $b \in \A_{< 2k } (\uparrow^q) $ and $c\in O(2k)$, then $b$ is primitive.
\end{lem}
As is known \cite{HM}, if a string link $T$ satisfies $\mathfrak{A }_{k+1} $ of the closure $\overline{T}$, $Z^t(T)=1+O(k -1 )$; thus, we can see that $ Z^t(T)_{< 2k } -1 \in \mathcal{P}^t_{< 2k }(q ) $ from property (III). The theorem is as follows:
\begin{thm}\label{thm224}
There is a linear isomorphism $R : \oplus_{j= k}^{2k-1} \mathcal{P}_{j}^t ( q ) \ra \Ker([\bullet,\bullet ]_{k,2k } \otimes \Q ) $ such that the following holds for any string link $T$ satisfying $\mathfrak{A }_{k+1} $ of the closure $\overline{T}$:
\begin{equation}\label{bbc2}R ( Z^t_{< 2k } (T) -1 ) = x_1 \otimes \lambda_1 + \cdots + x_q \otimes \lambda_q .
\end{equation} \end{thm}

Since Theorem \ref{thm24} implies that the left-hand side is equivalent to the Orr invariant, we have the following equivalence:
\begin{cor}\label{thm2244}
For any string link $T$ satisfying $\mathfrak{A }_{k+1} $ of the closure $\overline{T}$, the Orr invariant $ \theta_k( L,\tau)$ is equivalent to the tree reduction of the Kontsevich invariant of degree $< 2k $.
\end{cor}
\begin{exa}\label{kk}
As an example, we give a computation of the Boromean rings $6_2^3 $.
Let $PB(q+1) $ denote the pure braid group on $q+1$. Let $\sigma_i$ be the geometric braid formed by crossing the $i$-th string over the $(i+1)$-th one.
Consider the string link $T$ presented by $ \sigma_1^{-1} \sigma_2 \sigma_1^{-1} \sigma_2 \sigma_1^{-1} \sigma_2 $. Then $\overline{T} = 6_2^3 $.
We can easily verify $\mathfrak{A }_{3} $ and the expressions of the longitudes as
$$ \lambda_{i}= x_{i+2} x_i x_{i+1} x_i^{-1} x_{i+2}^{-1} x_i x_{i+1}^{-1} x_i^{-1}= [x_{i+2}, x_i x_{i+1} x_i^{-1}]. $$
where $ i \in \Z/3$. Thus, the $\mu$-invariant forms $ \sum_{i=1}^3 x_i \otimes [x_{i+2}, x_i x_{i+1} x_i^{-1}]$ modulo $F_4$. This value can be computed as something quantitative by using Magnus expansion $\mathcal{M}_4$ (see \S \ref{Sp34} for the definition of $\mathcal{M}_m$).
\end{exa}

\section{Proof of Theorem \ref{thm224}}
\label{Sp34}

As preparation, let us review the notion of group-like expansions and look at Example \ref{lie22}. Let $ \mathcal{I}_m \subset \Q \langle X_1, \dots, X_q \rangle $ be the both-sided ideal generated by polynomials of degree $\geq m. $ Consider the augmentation $\varepsilon : \Q \langle X_1, \dots, X_q \rangle\ra \Q $ with $\varepsilon (X_i)=1$, and a coproduct defined by $\Delta(X_i) =X_i \otimes 1 + X_i \otimes 1$. Then, the involution $S : \Q \langle X_1, \dots, X_q \rangle \ra \Q \langle X_1, \dots, X_q \rangle $ which sends $X_i$ to $-X_i$ makes it into a Hopf algebra. A {\it Magnus expansion} (modulo $O(m)$) is a group homomorphism $\theta : F \ra (\Q \langle X_1, \dots, X_q \rangle / \mathcal{I}_m )^{ \times }$, that satisfies $\theta (y )=1+ [ y] + O(2)$ for any $y \in Y$. Furthermore, a {\it group-like expansion} is a Magnus expansion $\theta$ satisfying $\Delta (\theta(y))= \theta(y) \otimes \theta(y)$ and $ \varepsilon (\theta(y)) =1$ for any $y \in Y$. For example, the homomorphism $\mathcal{M}_m : F \ra \Q \langle X_1, \dots, X_q \rangle/ \mathcal{I}_m $ which sends $ x_i $ to $1+X_i$ is not a group-like expansion, but a Magnus expansion.
\begin{rem}\label{ska}
We should mention some of the properties of these expansions (see \cite[Theorem 1.3]{Ka})

(1) Given another Magnus expansion $\theta'$, there is a ring automorphism $S_{\theta'}$ on $\Q \langle X_1, \dots, X_q \rangle / \mathcal{I}_m $ such that $ \theta_{\rm str} = S_{\theta'} \circ \theta'$.

(2) We have $ \theta (F_m)=0$, and $\theta $ induces an injection $\theta : F/F_m \ra \Q \langle X_1, \dots, X_q \rangle / \mathcal{I}_m . $ In fact, since the above $\mathcal{M}_m $ is injective, the injectivity inherits every $\theta$, by (1).
Furthermore, since the restriction on $F_j/F_{j+1} $ of $\mathcal{M}_{j+1} $ is additive by definition, that of $\theta $ is also an additive map.

(3) There is a Lie algebra isomorphism from $\mathcal{L}/ \mathcal{L}_{\geq m} $ to the subspace,
\begin{equation}\label{lie}
\mathcal{P} (F/F_m ) := \{ a \in \Q \langle X_1, \dots, X_q \rangle / \mathcal{I}_m \ | \ \Delta(a)= a \otimes 1 +1 \otimes a \ \}
\end{equation}
with the Lie bracket $[a,b]=ab - ba.$ The restricted image $\mathcal{M}_m (F_m /F_{m+1}) $ is contained in this $\mathcal{P} (F/F_m ) $.
\end{rem}

Next, the isomorphism \eqref{qq} below is related to another example of $\theta $ arising from the Kontsevich invariant. For an index pair $(i, j) \in \{1,\dots , q+1 \}^2 $, let $t_{i,j } \in \A_{1}^t (\uparrow^{q+1}) $ be the Jacobi diagram with only one edge connecting the $i$-th strand to the $j$-th one. Namely,

\begin{figure}[h]
\begin{center}
\begin{picture}(50,18)
\put(-63,-8){\Large $t_{ij}=$}
\put(-16,-7){\pc{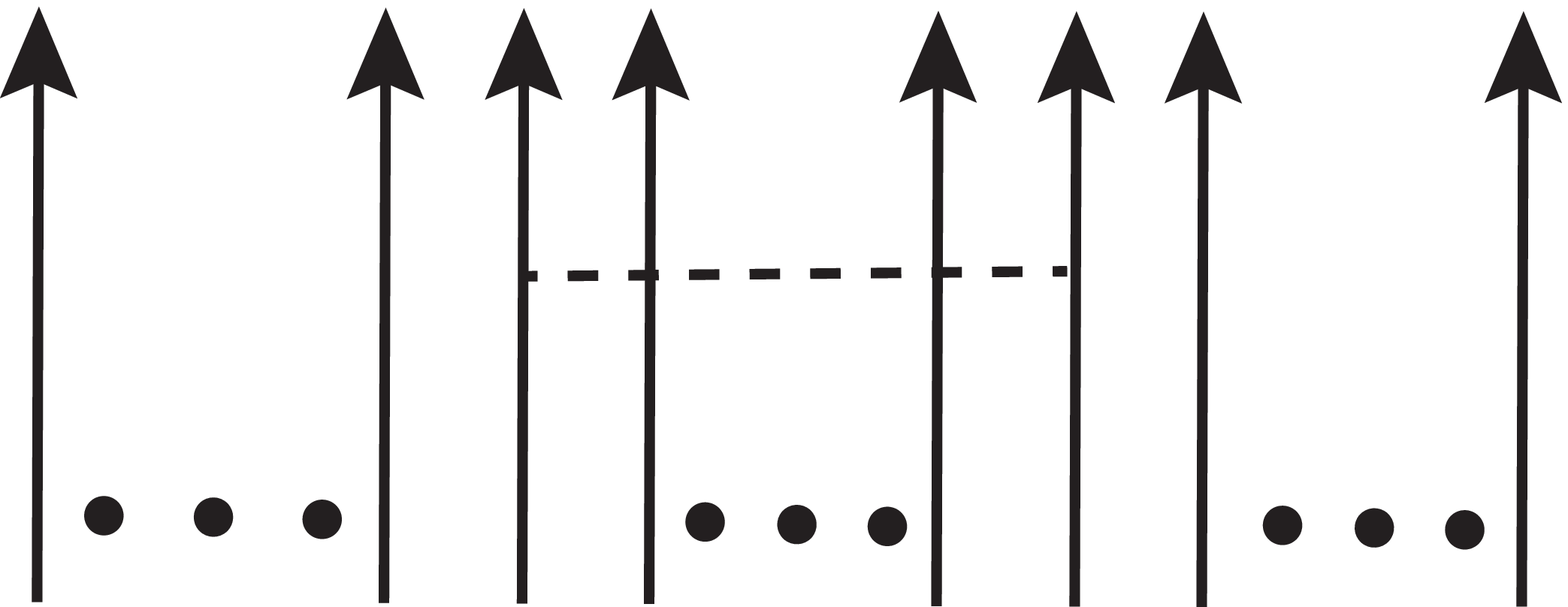}{0.144}}
\put(-15,-30){\large \ $ 1 \ \cdots i $ \ \ \ \ \ $j \cdots q+1 $}
\end{picture}
\end{center}
\end{figure}

\

\noindent
We denote $\A^{\rightleftharpoons} (\uparrow^{q},* ) $ by the subalgebra of $\A (\uparrow^{q+1}) $ generated by $t_{1, q+1}, t_{2, q+1} \dots, t_{q, q+1}$. Let $FI$ denote the framing independence relation in $\A^{\rightleftharpoons} (\uparrow^{q},* ) $, where any Jacobi diagram with an isolated chord on the same interval is equal to 0. Then, as shown \cite[(6.2)]{Mas2}, we can verify that the map,
\begin{equation}\label{qq}
\Q \langle X_1, \dots, X_q \rangle/ \mathcal{I}_m \lra \A^{\rightleftharpoons} (\uparrow^{q} ,* )/ (\A_{\geq m }^{\rightleftharpoons} (\uparrow^{q} ,* ) + FI),
\end{equation}
which sends $X_i$ to $t_{i, q+1}$ is a Hopf algebra isomorphism.

\begin{exa}
[{\cite[Proposition 6.2]{Mas2}}]\label{lie22}
We use notation on the pure braid group $PB(q+1) $ in Example \ref{kk}.
Let $\sigma_{j, q+1}\in PB(q +1 ) $ be $ \sigma_q \sigma_{q-1} \cdots \sigma_{j+1} \sigma_j^2 \sigma_{j+1}^{-1 } \cdots \sigma_j^{-1}$; see Figure \ref{tw}.
As is well known, there is a semi-direct product decomposition $ PB(q+1) \cong F(q) \ltimes PB(q)$, where $ F(q) $ is the free group generated by $\sigma_{1,q+1}, \dots, \sigma_{q,q+1}$, and $ PB(q) $ is embedded into $PB(q+1 )$ via $\beta \mapsto \beta \times \uparrow. $
Thus, any element $g$ of the free group $F (q) $ can be regarded as a pure-braid $ PB(q+1) \subset SL (q+1)$. Therefore, we can define $Z^t(g) \in \A^t (\uparrow^{q+1})$. As is shown \cite{Mas2}, $Z^t(g)$ lies in the subalgebra $\A^{\rightleftharpoons} (\uparrow^{q} ,* )$, and the composite,
\begin{equation}\label{qqqq}\theta^Z: F \stackrel{Z }{\lra }
\A^{\rightleftharpoons} (\uparrow^{q} ,* )/ (\A_{\geq m }^{\rightleftharpoons} (\uparrow^{q+1}) + FI) \stackrel{\eqref{qq}^{-1}}{\lra } \Q \langle X_1, \dots, X_q \rangle/ \mathcal{I}_m,
\end{equation}
turns out to be a group-like expansion (Moreover, it was shown to be a ``special expansion").
\end{exa}
Before turning back to Theorem \ref{thm224}, we should mention \eqref{qqqq2} from \cite[Corollary 12.2]{HM}. For a pure braid $\sigma \in PB(q)$, the $\ell$-th longitude, $\lambda_{\ell} \in F (q)$, of $\sigma$ is equal to $ ( \sigma \times 1 )^{-1} \beta_{\ell} D_{\ell } (\sigma) \beta_{\ell} ^{-1}$ in $ PB(q+1)$. Here, $\beta_{\ell}$ is a braid of the form $ \sigma_1 \sigma_2 \cdots \sigma_{\ell-1} \in B_n .$ Thus, the $\theta^Z(\lambda_{\ell} )$ can be computed as
\begin{equation}\label{qqqq2} \theta^Z (\lambda_{\ell})= Z^t( \sigma \times 1 )^{-1} Z^t(\beta_{\ell}) \Delta_{\ell } (Z^t (\sigma)) Z^t(\beta_{\ell})^{-1} \in \A^{\rightleftharpoons} (\uparrow^{q} ,* ). \end{equation}

\begin{figure}[tpb]
\begin{center}
\begin{picture}(110,74)
\put(16,37){\pc{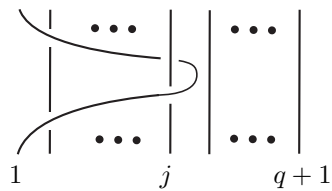}{0.200}}
\put(18,3){$1$}
\put(75,3){$j $}
\put(118,3){$q+1$}
\end{picture}
\end{center}
\vskip -1.7pc
\caption{\label{tw} The pure braid $\sigma_{j,q+1}$ }
\end{figure}

\begin{proof}[Proof of Theorem \ref{thm224}]
Inspired by \eqref{qqqq2}, for $\ell \leq q $, we set up a homomorphism,
$$ \Upsilon_\ell^{(j)} : \mathcal{P}^{t}_j (\uparrow^{q} ) \lra \A^{\rightleftharpoons}_{\leq j} (\uparrow^{q} ,* ) /FI \cong \Q \langle X_1, \dots, X_q \rangle / \mathcal{I}_{j+1} ; \ a \longmapsto ( (a+1) \times 1 )^{-1} Z^t(\beta_{\ell}) \Delta_{\ell } (a +1) Z^t(\beta_{\ell})^{-1}. $$
Furthermore, we consider the linear homomorphism,
$$ \Upsilon^{(j)} : \mathcal{P}^{t}_j (\uparrow^{q} ) \lra F/F_2 \otimes_{\Z } \Q \langle X_1, \dots, X_q \rangle / \mathcal{I}_{j+1} ; \ \
a \longmapsto \sum_{\ell : \ 1 \leq \ell \leq q }x_i \otimes \Upsilon_\ell^{(j)} (a). $$
We will show the injectivity of $\Upsilon^{(j)} $ and that the image is equal to
\begin{equation}\label{qqqq5} \{ \sum_{1 \leq \ell \leq q} x_{\ell} \otimes S_{\theta^Z}(a_{\ell}) \in F/F_2 \otimes (\mathcal{I}_{j} / \mathcal{I}_{j+1})\ | \ \Delta(a_i) = a_i \otimes a_i, \ \sum_{1 \leq \ell \leq q} a_{\ell}X_{\ell} - X_{\ell} a_{\ell}=0 \ \}. \end{equation}
From Remark \ref{ska} (3), this subspace \eqref{qqqq5} can be identified with the kernel of $[\bullet, \bullet]_{j,j+1 } : F/F_2 \otimes \mathcal{L}_{j}/ \mathcal{L}_{j+1} \ra \mathcal{L}_{j+1}/ \mathcal{L}_{j+2 }$. Thus, the rank of \eqref{qqqq5} is $q N_j - N_{j+1}. $ As is known (see \cite{Le,HM}), if some $ (\alpha_1,\dots, \alpha_q ) \in F_j/F_{j+1} $ satisfies $[x_1, \alpha_1 ]\cdots [x_q, \alpha_q] =1 \in F/F_{j+1}$, then there exists a string link $T_{\alpha}$ with $ \bar{\mu}$-invariants $( \lambda_1, \dots, \lambda_q)\in (F_j)^q$ such that $ \lambda_j \equiv \alpha_j $ mod $F_{j+1}$. Therefore, by \eqref{qqqq2}, the image of $\Upsilon^{(j)} $ is generated by $\Upsilon^{(j)} (Z^t_{\leq j} ( T_{\alpha})-1 )= \sum_{i=1}^{q } x_\ell \otimes \theta^Z (\lambda_\ell ) $, where $\alpha $ runs over a basis of $\Ker[\bullet, \bullet]_{j,j +1 }. $ Since the rank of $ \mathcal{P}^{t}_j ( \uparrow^q )$ is $q N_j - N_{j+1}$, $\Upsilon^{(j)} $ must be injective, and the image is \eqref{qqqq5}, as required.

To complete the proof, we will construct the isomorphism $R$ and show the equality \eqref{bbc2}.
Consider
$$\mathrm{Id}_{\Z^q} \otimes \theta^Z :
\Z^q \otimes_{\Z} F_k /F_{2k } \lra F/F_2 \otimes_{\Z } \Q \langle X_1, \dots, X_q \rangle / \mathcal{I}_{2k},$$
which is an injective homomorphism, by Remark \ref{ska} (2). The image of $\Ker[\bullet, \bullet]_{j,j+1} $ is contained in \eqref{qqqq5}. We define the isomorphism $R : \Ker([\bullet,\bullet ]_{k,2k-1 } \otimes \Q ) \ra \oplus_{j=k}^{2k-1} \mathcal{P}^t_{j } ( q ) $ as a linear extension of $ \mathrm{Id}_{\Z^q} \otimes (\mathrm{id}_{\Z^q} \otimes \theta^Z ) \circ\oplus_{j=k}^{2k-1}( \Upsilon^{(j)} )^{-1} )$. 

We now prove the equality \eqref{bbc2} where the string link $T$ is a pure braid. Noting that $Z_{<2k } (T)-1$ is primitive by Lemma \ref{lem122}, we compute $\Upsilon^{(\ell )}( Z^t_{< 2k } (T)-1) $ as
\[ \sum_{\ell=1}^{q} \bigl( x_\ell \otimes \sum_{j=k}^{2k-1} (\Upsilon^{(j)}_{\ell})( Z^t_{< 2k } (T)-1)\bigr) = \sum_{\ell=1}^{q} x_\ell \otimes ( \theta^Z(\lambda_\ell ) )= \sum_{\ell=1}^{q} \theta^Z( x_\ell) \otimes ( \theta^Z(\lambda_\ell ) )
\]
which is immediately deduced from from \eqref{qqqq2}.
By the injectivity of $\theta^Z$, this equality implies the desired \eqref{bbc2}.

Finally, we will prove \eqref{bbc2} for any string link $T$ satisfying $\mathfrak{A }_{k+1} $ of the closure $\overline{T}$.
As implicitly shown in the proofs of \cite[Propositions 10.6 and Theorem 6.1]{HM},
there is a pure braid $\sigma$ such that the Milnor invariants of $T$ and $\sigma$ in $F_k/F_{2k}$ are equal. Thus, the invariant of $ T \sigma^{-1}$ is zero. Thus, \cite[Theorem 6.1]{HM} immediately implies $ Z^t( T \sigma^{-1}) = 1 +O(2k)$,
which leading to $ Z^t( T ) =Z^t( \sigma)$ modulo $ O(2k)$ by Lemma \ref{lem122} (1).
Hence, $ R (Z^t( T ) -1 ) =R ( Z^t( \sigma)-1) $ is equal to $\sum_{i=\ell}^q x_i \otimes \lambda_\ell $ by the above paragraph, as desired. It completes the proof.
\end{proof}

\section{Proof of Theorem \ref{thm24}}
\label{Sher33}

\subsection{Review of the infinitesimal Morita-Milnor homomorphism, and tree reduction}
\label{Sher}
As a preliminary to prove Theorem \ref{thm24}, we review some of the results in \cite{Mas,IO,K}.

We will start by briefly reviewing the Lie algebra homology of $F/F_k$. Let $H$ be the $\Q$-vector space of rank $q$ with basis $X_1, \dots, X_q$, i.e., $H= \mathrm{Span}_{\Q} \langle X_1, \dots, X_q \rangle. $ Let $\mathfrak{L}$ be the free Lie algebra generated by $H$. This $\mathfrak{L}_{\geq k}$ is the subspace generated by the commutator of length $\geq k$. We have the quotient Lie algebra $\mathfrak{L} / \mathfrak{L}_{\geq k} $. Then, {\it the Koszul complex} of $\mathfrak{L} / \mathfrak{L}_{\geq k} $ is the exterior tensor algebra $\Lambda^{*} (\mathfrak{L} / \mathfrak{L}_{\geq k} ) $ with the boundary map $\partial_n: \Lambda^{n} (\mathfrak{L} / \mathfrak{L}_{\geq k}) \ra \Lambda^{n-1} (\mathfrak{L} / \mathfrak{L}_{\geq k}) $ given by
$$ \partial_n(h_1 \wedge \cdots \wedge h_n ) = \sum_{i <j} (-1)^{i+j} [h_i, h_j]\wedge h_1
\wedge \cdots \wedge \check{h}_{i} \wedge \cdots \wedge \check{h}_{j} \wedge \cdots \wedge h_n. $$
Later, we will use the known isomorphism $H_n( \Lambda^{*} (\mathfrak{L} / \mathfrak{L}_{ \geq k} ) ) \cong H_n(F/F_{k}; \Q )$, which is called Pickel's isomorphism; see, e.g., \cite{IO, Mas,SW}.

Next, let us review Jacobi diagrams. A {\it Jacobi diagram} is a uni-trivalent graph whose univalent vertices are labeled by one of $\{ 1,2,\dots, q\}$, where each trivalent vertex is oriented. Consider the graded $\Q$-vector space generated by Jacobi diagrams, where the degree of such a diagram is half the number of vertices. Let $\mathcal{J}(q)$ be the quotient space subject to the AS and IHX relations, and let $\mathcal{J}^t (q) \subset \mathcal{J}(q)$ be the subspace generated by Jacobi diagrams which are simply connected. As a diagrammatic analogy of the Poincar\'{e}-Birkhoff-Witt theorem, we can construct a graded vector isomorphism,
\begin{equation}\label{bbcc}\chi_q : \mathcal{P}^t_n (\uparrow^q) \cong \mathcal{J}^t_n (q). \end{equation}
See, e.g., \cite[\S 5.7]{CDM} for details.
\begin{rem}\label{ll}
Theorem \ref{thm224} implies a generalization of the main theorem 6.1 of \cite{HM}.
In fact, recalling from Theorem \ref{thdr} (II) the $k$-th summand of $\theta_k (L,\tau)$ is the Milnor invariant of length $k$,
the equivalence between the Milnor invariant and $\chi_q(Z^t_{k}(T)-1 )$ coincides with the theorem 6.1 of \cite{HM} exactly.
\end{rem}

Furthermore, we consider the subspace $ \mathcal{J}_n (q, 0)$ of $ \mathcal{J}_n (q+1) $ generated by tree diagrams in which the label $q+1$ occurs exactly once. Given such a diagram $J \in \mathcal{J}_n (q, 0) $, all of its univalent vertices apart from one labeled by $r$ defines $\mathrm{comm}(J)$ in $F_n /F_{n+1}$ in a canonical way. For example,
\begin{figure}[h]
\begin{center}
\begin{picture}(50,74)

\put(-38,63){\large \ \ \ \ \ $x_1 \ \ \ \!\! x_2 \ \ \!\! x_3 \ x_4 $}

\put(8,19){\large \ \ \ $q+1 $}

\put(-63,38){\large comm $\Bigl(\ \ \ \ \ \ \ \ \ \ \ \ \ \ \ \ \ \ \Bigr) = [ x_1,[[x_2,x_3], x_4 ]] \in F_4 . $}
\put(-16,37){\pc{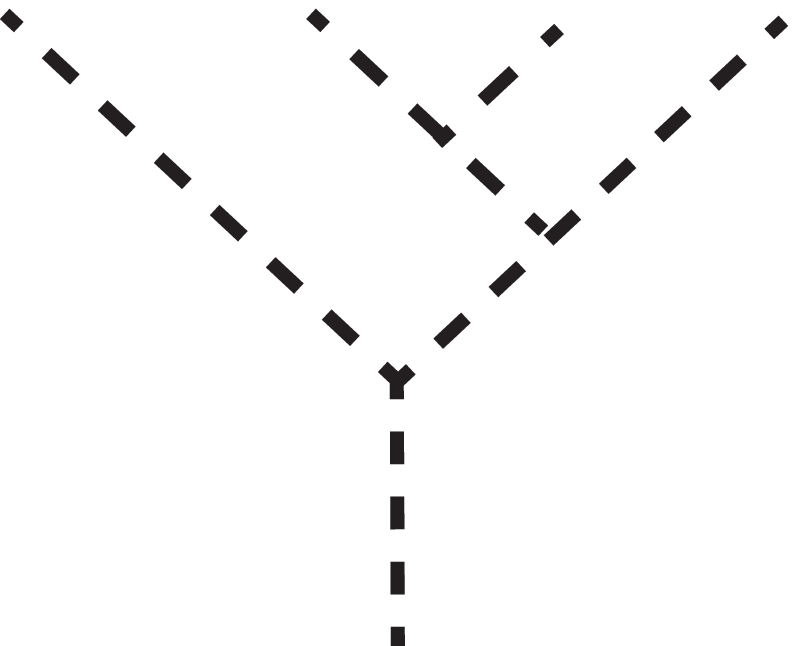}{0.2114}}
\end{picture}
\end{center}
\end{figure}

\vskip -2.7pc

\noindent
Then, we can easily see that the following correspondence is an isomorphism,
$$ \mathrm{comm}: \mathcal{J}_n ( q, 0) \lra F_n /F_{n+1}\otimes \Q . $$

Next, we will review the infinitesimal Morita-Milnor homomorphism, $M_k$, defined in \cite[\S 5.1]{K} (cf. the infinitesimal Morita map in \cite{Mas2}). For this, a string link $T \in SL (q)$ is {\it of degree $k$} if the closure $\overline{T}$ satisfies Assumption $\mathfrak{A}_{k+1} $. Let $SL (q)_k $ be the subset consisting of string links of degree $k$. Then, we have a filtration,
$$SL (q)=SL (q)_0 \supset SL (q)_1 \supset \cdots \supset SL (q)_k \supset \cdots. $$
Further, we recall the $\bar{\mu}$-invariants $\lambda_{\ell} \in F_k/F_{2k -1 }$ and expand them as $\lambda_{\ell} = \lambda_{\ell}^{(k)} + \cdots + \lambda_{\ell}^{(2k -2 )}$ with $\lambda_{\ell}^{(j)} \in F_j/F_{j+1} $ according to $F_k/F_{2k-1} \cong \oplus_{j=k}^{2k-2} F_j/F_{j+1} $. Recall that, if $ L \in SL (q)_k $ is of degree $k$, the $\bar{\mu}$-invariants $ \lambda_\ell $ are contained in $F_k $. Then, identifying $F_{j}/F_{j+1} \otimes \Q$ with $ \mathfrak{L}_j/\mathfrak{L}_{j+1}$ as a $\Q$-vector space, let us consider a 2-form,
\begin{equation}\label{bbc9}
\sigma_L := \sum_{j=k}^{ 2k-2 } \sum_{\ell =1}^{q} X_\ell \wedge \lambda_\ell^{(j)} \in \Lambda^2 (\mathfrak{L}/\mathfrak{L}_{\geq 2k -1 } ) .
\end{equation}
As is shown in \S 5.1 of \cite{K}, there exists $t_L \in \Lambda^3 (\mathfrak{L}/\mathfrak{L}_{\geq 2k -1} )$ satisfying $\partial_3(t_L) =\sigma_L $.
Since $ \lambda_\ell \in F_k$ by Assumption $\mathfrak{A}_{k+1}$, the 2-form $ \sigma_L $ reduced in $\Lambda^2 (\mathfrak{L}/\mathfrak{L}_{ \geq k} ) $ is zero. Hence, $t_L$ is a 3-cycle. It has been shown \cite[Lemma 5.1.2]{K} that the homology 3-class $[\{ t_L\}] \in H_3 (\mathfrak{L}/\mathfrak{L}_{ \geq k} )$ is independent of the choice of $t_L$. To summarize, we have a map,
$$M_{k} : SL(q)_k \lra H_3 (\mathfrak{L}/\mathfrak{L}_{ \geq k} ); \ \ L \longmapsto [\{ t_L\}] .$$
Kodani \cite[\S 5]{K} showed that $M_{k}$ is a monoid homomorphism and its kernel is $SL(q)_{2k-1}$.

Now let us review the tree description of the third homology \cite{Mas}. Fix a tree diagram $J \in \mathcal{J}^t_j (q)$. For each trivalent vertex $r$ of $J$, $J$ is the union of the three tree diagrams rooted as $r$. We denote the three tree diagrams by $J_r^{(1)}$, $J_r^{(2)}$, and $J_r^{(3)}$, by clockwise rotation. Here, the numbering 1,2,3 is according to the cyclic ordering of $r$. Furthermore, if we label a univalent $r$ by $q+1$, each $J_r^{(j)}$ can be regarded as an element in $\mathcal{J}^t(q,0)$. Then, the {\it fission map} $\phi: \mathcal{J}^t_j (q) \ra \Lambda^{3} (\mathfrak{L} / \mathfrak{L}_{ \geq k} ) $ is defined by
$$ \phi(J)= \sum_{r: \mathrm{trivalent \ vertex \ of \ }J } \mathrm{comm}(J_r^{(3)} )\wedge \mathrm{comm}(J_r^{(2)} )\wedge \mathrm{comm}( J_r^{(1)}). $$
\begin{thm}
[\cite{Mas}]\label{thma1} If $ k \leq j \leq 2k - 2 $, this $\phi(J)$ is a 3-cycle. Moreover, the fission map gives rise to a linear isomorphism,
$$ \Phi : \bigoplus_{j=k }^{2k-2} \mathcal{J}^t_j ( q ) \stackrel{\sim}{\lra} H_3( \mathfrak{L} / \mathfrak{L}_{ \geq k} ) . $$
\end{thm}

Next, let us review the isomorphism \eqref{bbc} below. For a Jacobi diagram $J^{(j)} \in \mathcal{J}_j^t (q) $, we define $ \eta_{j} ( J^{(j)})$ to be the sum,
\begin{equation}\label{bbc6} \sum_{v: \mathrm{univalent \ vertex \ of \ }J^{(j)}} [ x_{\mathrm{col}(v)}] \otimes \mathrm{comm}(J_v^{(j)}) \in (F/F_2) \otimes_{\Q} (\mathcal{L}/\mathcal{L}_{> j} ), \end{equation}
where ${\mathrm{col}(v)} \in \{ 1, \cdots, q\} $ is the label on $v$, and $J_v^{(j)} \in \mathcal{J}_j^t (q,0) $ is the labelled tree obtained by replacing the label on $v $ with $q+1$. Taking the bracket $ [ -, -]:F/F_2 \otimes \mathcal{L}_{ j} \ra \mathcal{L}_{ j+1} $, which $(x,y)$ sends $xyx^{-1} y^{-1}$, we can easily verify that the sum \eqref{bbc6} lies in the kernel. Then, this $\eta_j$ defines a linear homomorphism,
\begin{equation}\label{bbc}\eta_{j} : \mathcal{J}_j^t (q) \lra \Ker([ -, -]: F/F_2 \otimes \mathcal{L}_{ j}/\mathcal{L}_{ j+1} \ra\mathcal{L}_{ j+1}/ \mathcal{L}_{ j+2} ), \end{equation}
which is known to be an isomorphism; see \cite{Le2}. We denote by $\eta $ the sum of the isomorphisms $\oplus_{j=k}^{2k -2 } \eta_{j} $ for short.

\subsection{Proof of Theorem \ref{thm224}}
\label{Sp2}
Then, the following lemma relates to the 3-class $[t_L]$ from tree diagrams:
\begin{lem}\label{oo2}
For $T \in SL(q)_{k}$, the composite $ \Phi \circ \eta^{-1} (\sum_{\ell=1}^q x_{\ell} \otimes \lambda_{\ell}) $ coincides with $M_k(T)=[t_L] $.
\end{lem}
\begin{proof}
Let $J^{(j)} \in \mathcal{J}^t_j(q)$ be $(\eta)^{-1}_j ( \sum_{i=1}^qx_i^{(j)}\otimes \lambda_i^{(j)})$. As is shown in \cite[Lemma 1.2]{Mas}, the composite $\partial_3(\phi ( J^{(j)})) $ is formed as
$$ \partial_3(\phi ( J^{(j)})) = \sum_{v: \mathrm{univalent \ vertex \ of \ }J^{(j)}} X_{\mathrm{col}(v)} \wedge \mathrm{comm}(J^{(j)}) \in \wedge^2 (\mathcal{L}/\mathcal{L}_{ \geq 2k} ). $$
Compared with \eqref{bbc6}, this $ \partial_3(\phi ( J^{(j)})) $ is equal to $\sigma_L \in \wedge^2 (\mathcal{L}/\mathcal{L}_{ \geq 2k} ) $ by definition. Thus, letting $t_L$ be $ \sum_{j=k}^{2k-2} \phi (J^{(j)}) $, we get a 3-chain $t_L$ satisfying $ \partial_3(t_L) =\sigma_L $, leading to the desired coincidence.
\end{proof}
We can prove Theorem \ref{thm24} by showing that the 3-class $[t_L]$ is equal to the Orr invariant of degree $k $. For this, we will have to prove the following lemma.
\begin{lem}\label{a2}
From the identification $H_3(\mathcal{L}/\mathcal{L}_{\geq k}) \cong H_3^{\rm gr}(F/F_k;\Q )$, the homomorphism $ M_k $ is equal to a map which sends the 3-class $t_L$ to $ \mathfrak{H} \circ \theta_k (\overline{T},\tau).$
\end{lem}
Before proving this lemma, we now prove Theorem \ref{thm24}.
\begin{proof}
[Proof of Theorem \ref{thm24}] The proof of (I) immediately follows from Lemmas \ref{oo2} and \ref{a2}.

Now let us show (II). Fix a basis $\{ b_s \} $ of $H_3(F/F_k;\Z ) $, where $s$ ranges over $1 \leq s \leq \mathrm{rk}H_3(F/F_k;\Z ) $. Thanks to Theorem \ref{thdr} (III), we can choose a string link $L_{s} $ such that $ \mathfrak{H} (\theta_k ( \overline{L_s} ,\tau))= b_s$ and the component of $ \theta_k ( \overline{L_s} ,\tau) $ in the kernel $ \Ker \mathfrak{H}$ is zero. Then, from (1), the $\bar{\mu}$-invariant of $L_s $ in $F_k /F_{2k}$ is $ \eta \circ \Phi^{-1} (b_s ) + \sum_{j=1}^q x_j \otimes \lambda_{j,s}^{(2k)} $ for some $ \lambda_{j,s}^{(2k)} \in F_{2k-1}/F_{2k} \otimes \Q $. Recalling the isomorphism $\Ker \mathfrak{H} \otimes \Q \cong \Q^{q N_{2k-1} - N_{2k} } $, we fix the isomorphism $\iota : \Q^{q N_{2k-1} - N_{2k} } \ra \Ker[ \bullet, \bullet ]_{2k-1,2k}$ and define the bijection,
$$ \overline{\Psi}_k : \bigl( H_3(F/F_k;\Z ) \oplus \Ker \mathfrak{H}\bigr) \otimes \Q \lra \Ker[ \bullet, \bullet ]_{k,2k }, $$
by setting
$$\overline{\Psi}_k \bigl(( b_s , a) \otimes r \bigl) = r \bigl(\eta \circ \Phi^{-1 } (b_s )- (x_1 \otimes \lambda_{1,s}^{(2k)} + \cdots + x_q \otimes \lambda_{q,s}^{(2k)}) + \iota( a) \bigl), \ \ \ \ r \in \Q. $$

Let $ \overline{\Phi}$ be $\overline{\Psi}_k ^{-1}$. Then it suffices to show the desired equality $x_1\otimes \lambda_1+ \cdots + x_q\otimes \lambda_q = \overline{\Psi}_k \bigl( \theta_k(\overline{T} , \tau)\bigr) $ for any string link $T$ of degree $k$. Suppose $\mathfrak{H} \circ \theta_k (\overline{T} , \tau) = \sum_s \alpha_s b_s $ for some $\alpha_s \in \Z$. Then, the stucking $ T \sharp (\sharp_s (L_s)^{-\alpha_s} )$ is of degree $k+1.$ Let $\mathfrak{H}'$ be the Hurewicz map $\pi_3(K_{k+1}) \ra H_3(F/F_{k+1}) $. Using the additivity of $\theta_k $ and of the $\bar{\mu}$-invariants, we can compute $\overline{\Psi}_k \bigl( \theta_k(\overline{T} , \tau)\bigl) $ restricted to the kernel of $[\bullet, \bullet]: F/F_ 2 \times F_{2k-1}/F_{2 k} \ra F_{2k}/F_{ 2k +1 } $ as
\[\overline{\Psi}_k \bigl( \theta_k(\overline{T} , \tau)\bigr) = \overline{\Psi}_k \bigl( \theta_k(\overline{T}\sharp (\sharp_s (L_s)^{-\alpha_s} ) , \tau) + \theta_k(\sharp_s (L_s)^{\alpha_s} , \tau)\bigr) \]
\[= \overline{\Psi}_k \bigl( \theta_k(\overline{T}\sharp (\sharp_s (L_s)^{-\alpha_s} ) , \tau)\bigr) + \overline{\Psi}_k \bigl( \theta_k(\sharp_s (L_s)^{\alpha_s} , \tau)\bigr) \]
\[=\overline{\Psi}_{k+1} \circ \mathfrak{H}' \bigl( \theta_{k+1}(\overline{T}\sharp (\sharp_s (L_s)^{-\alpha_s} ) , \tau) \bigr) + \overline{\Psi}_k \bigl( \sum_s \alpha_s (b_s,0) \bigr) \]
\[=\sum_{j=1}^q X_i \otimes \lambda_j( T \sharp (\sharp_s (L_s)^{-\alpha_s} ) + \sum_s \alpha_s (x_1 \otimes \lambda_{1,s}^{(2k)} + \cdots + x_q \otimes \lambda_{q,s}^{(2k)}) = \sum_{j=1}^q x_i \otimes \lambda_j( T ),\]
as desired. Here, the third equality is obtained from (II), and the last is done from additivity of $\lambda $.
\end{proof}

\begin{proof}
[Proof of Lemma \ref{a2}] First, we set up some complexes and chain maps. In what follows, we consider only complexes over $\Q$ and omit writing the coefficients $\Q$. Let $( C_*^{\rm gr} (F/F_k), \partial_* )$ be the non-homogenous group complex of $F/F_k$. Then, on the basis of Suslin and Wodzicki's paper \cite{SW}, Massuyeau (see the proof of \cite[Proposition 4.3]{Mas}) showed the natural existence of a chain map $ \kappa: \bigwedge^*( \mathcal{L} / \mathcal{L}_{\geq k}) \ra C_*^{\rm gr} (F/F_k) $ that induces an isomorphism on the homology. Thus, it is enough for us to show that the 3-cycle $\kappa( t_L)$ is equivalent to the pushforward $(f_k)_* [S^3 \setminus \overline{T},\partial (S^3 \setminus \overline{T})]$.

To do so, we can study the 3-cycle $t_L$ from the viewpoint of the group complex. For $\ell \leq q$, let $K_\ell$ be the abelian subgroup of $F$ generated by the meridian-longitude pair $(\mathfrak{m}_{\ell} , \mathfrak{l}_{\ell} )$. Let us consider the commutative diagram,
$$
{\normalsize
\xymatrix{
0 \ar[r] & \oplus_{\ell=1}^q C_3^{\rm gr} (K_\ell ) \ar[r]^{\iota_3} \ar[d]_{\partial_3} & C_3^{\rm gr } ( F ) \ar[d]_{\partial_3} \ar[r]^{\!\!\!\!\!\!\!\! \!\!\!\! P_3} &C_3^{\rm gr } (F, K_1 \cup \cdots \cup K_q ) \ar[d]_{\partial_3} \ar[r] & 0 & (\mathrm{exact}) \\
0 \ar[r] & \oplus_{\ell=1}^q C_2^{\rm gr} (K_\ell ) \ar[r]^{\iota_2} & C_2^{\rm gr } (F ) \ar[r] ^{\!\!\!\!\!\!\!\! \!\!\!\! P_2} &C_2^{\rm gr } (F, K_1 \cup \cdots \cup K_q ) \ar[r] & 0& (\mathrm{exact}),
}}$$
where the right-hand sides are defined as the cokernel of $\iota_*$. In the subcomplex $C_2^{\rm gr} (K_\ell ) $, the cross product $\mathfrak{m}_{\ell} \times \mathfrak{l}_{\ell}$ is a 2-cycle that generates $ H_2^{\rm gr} (K_\ell ) \cong \Z $. Let $\tau_L$ be a 2-cycle $\sum_{\ell=1}^q \mathfrak{m}_{\ell} \times \mathfrak{l}_{\ell} \in \oplus_{\ell=1}^q C_2^{\rm gr} (K_\ell )$.
Accordingly, since $H_2^{\rm gr}(F)=0$, we can choose a 3-cycle $\eta_L $ in $ C_3^{\rm gr } (F, K_1 \cup \cdots \cup K_q )$ such that $\delta_*(\eta_L )= \tau_L.$

Next, we will examine the diagrams subject to $F_k$ and $F_{2k}$ with regard to their functoriality:
$$
{\normalsize
\xymatrix{
0 \ar[r] & \oplus_{\ell=1}^q C_2^{\rm gr} (K_\ell ) \ar[r] \ar[rdd] \ar[d]_{\partial_3} & C_3^{\rm gr} ( F/F_{2k -1 } ) \ar[d]_{\partial_3} \ar[r]^{P^{\rm gr}} \ar[rdd] &C_3^{\rm gr} (F/F_{2k -1 }, \cup_{\ell }K_\ell ) \ar[d]_{\partial_3} \ar[r] \ar[rdd] & 0 & & \\
0 \ar[r] & \oplus_{\ell=1}^q C_2^{\rm gr} (K_\ell ) \ar[r] \ar[rdd] & C_2^{\rm gr} ( F/F_{2k-1 }) \ar[r] \ar[rdd] &C_2^{\rm gr} (F/F_{2k -1 }, \cup_{\ell } K_\ell) \ar[r] \ar[rdd] & 0& & \\
& 0 \ar[r] & \oplus_{\ell=1}^q C_3^{\rm gr} (K_\ell ) \ar[d]_{\partial_3} \ar[r] & C_3^{\rm gr} ( F/F_k ) \ar[d]_{\partial_3} \ar[r]_{ \!\!\!\!\!\!\!\!\!\!\!\! P^{\rm gr}} &C_3^{\rm gr} (F/F_k, \cup_{\ell }K_\ell ) \ar[d]_{\partial_3} \ar[r] & 0 \\
& 0 \ar[r] & \oplus_{\ell=1}^q C_2^{\rm gr} (K_\ell ) \ar[r] & C_2^{\rm gr} ( F/F_k ) \ar[r] &C_2^{\rm gr} (F/F_k, \cup_{\ell } K_\ell ) \ar[r] & 0.
}}$$
Here, the horizontal arrows are exact, and the slanting ones are the maps induced from the projection $F/F_{2k -1 } \ra F/F_k. $ Since the above quasi-isomorphism $\kappa $ was constructed from the projective resolution of the augmentation $\varepsilon : \Q[G] \ra \Q$, this $\kappa $ replaces the wedge product $\wedge $ by the cross product $\times$. Therefore, recalling the isomorphism $\phi^*$ in \eqref{dd}, the 2-cycle $\phi_*^{-1}(\tau_L) $ modulo $F_{2k}$ is exactly equal to $\sigma_L$ in \eqref{bbc9}. Thus, the 3-cycle $t_L \in C_3^{\rm gr}(F/F_{2k -1 } )$ satisfies $P^{\rm gr} (t_L) = \phi_*^{-1} ( \eta_L) \in H_3^{\rm gr}(F/F_{k} )$.

Finally, we give a relation to the link complement $S^3 \setminus \overline{T}$. Let $E$ denote $S^3 \setminus \overline{T}$. Similarly, let us consider the long exact sequence on the cellular homology:
$$ 0 \ra H_3^{\rm cell } (E, \partial E;\Q ) \stackrel{\delta}{\lra} H_2^{\rm cell } ( \partial E;\Q ) \lra H_2^{\rm cell } (E;\Q )\lra H_2^{\rm cell } (E, \partial E;\Q ).$$
Here are some well-known facts from knot theory: The first term is $\Q$ generated by the fundamental 3-class $[E,\partial E]$, and the second is $\Q^{q}$ generated by the cross products $\mathfrak{m}_{\ell} \times \mathfrak{l}_{\ell}$. Further, the sum $\sum_{\ell =1}^q \mathfrak{m}_{\ell} \times \mathfrak{l}_{\ell}$ is zero in $ H_2^{\rm cell } (E ;\Q )$ by relation \eqref{le}. Thus, we have a cellular 3-chain $ \overline{b_L} \in C_3^{\rm cell } (E, \partial E;\Q )$ such that $ (P_3)_*[ \overline{b_L} ] = [E,\partial E]$ in the relative $C_3^{\rm cell } (E, \partial E;\Q )$ and $\partial_3( \overline{b_L} ) = \sum_{\ell =1}^q \mathfrak{m}_{\ell} \times \mathfrak{l}_{\ell}.$ In particular, letting $I: [0,1]^3 \setminus T \ra S^3 \setminus \overline{T}$ be the inclusion, we have $[E,\partial E] = I_* ( \eta_L) $. Notice that $f_k \circ I_* : \pi_1( [0,1]^3 \setminus T ) \ra F/F_k $ is equal to the reduction $F \ra F/F_k. $ 
Hence, noting $ (f_k)_* \circ I_* = \phi_*^{-1}$, we obtain the computation,
$$ (f_k)_* [E,\partial E] =(f_k)_* \circ I_*(\eta_L) = \phi_*^{-1}(\eta_L)= P^{\rm gr}(t_L) \in H_3^{\rm gr} (F/F_k, \cup_{\ell}K_{\ell}). $$
Since $ K_{\ell} $ modulo $F_k$ is isomorphic to $ \Z $ from Assumption $\mathfrak{A}_{k+1}$, $P^{\rm gr}$ induces $H_3^{\rm gr} (F/F_k ) \cong H_3^{\rm gr} (F/F_k, \cup_{\ell }K_\ell )$. Hence, $ t_L$ is equivalent to the pushforward $(f_k)_* [E,\partial E]$, as desired.
\end{proof}

\appendix
\section{HOMFLYPT polynomials and Orr invariants}
\label{SHigher}
Let $T$ be a string link and $(\overline{T}, \tau)$ be the associated based link. According to Theorem \ref{thm24}, the computation of the Orr invariant of $(\overline{T}, \tau)$ is equivalent to the $\bar{\mu}$-invariant of $ T$ of degree $< 2k$. In general, it is hard to get a presentation of longitudes $\lambda_{\ell}$ as a word of $x_1, \dots, x_q $ (However, if $T$ is a pure braid, we can easily get such a presentation). Furthermore, in quantum topology, it is natural to ask what finite type invariants recover the $\bar{\mu}$-invariants. As a solution, we will show that the result of Meilhan-Yasuhara \cite{MY} give a computation of the $\bar{\mu}$-invariants from HOMFLYPT polynomials, without having to write longitudes (Theorem \ref{ooth2}).

To describe the result, we should recall the HOMFLYPT polynomial and some of its properties. {\it The HOMFLYPT polynomial} $P ( L; t, z) \in \Z [t^{\pm 1} , z^{\pm 1}]$ of an oriented link $L \subset S^3 $ is defined formulas as follows:
\begin{enumerate}[(I)]
\item Concerning the unknot $U$ in $S^3$, the polynomial $ P ( U; t, z)$ is $1$.
\item The skein relation $t \cdot P ( L_+ ; t, z) + t^{-1} \cdot P ( L_- ; t, z) = z \cdot P ( L_0 ; t, z)$ holds, where $L_-$, $ L_+$, and $L_0$ are links formed by crossing and smoothing changes on a local region of a link diagram, as indicated in Figure \ref{t321}.
\end{enumerate}
Furthermore, given a $q $-component link $L$, we can expand the HOMFLYPT polynomial as
$$ P ( L ; t, z) = \sum_{k=1}^N P_{2k-1-q} (L;t ) z^{2k-1-q}, $$
for some $N \in \mathbb{N}$, where $P_{2k-1-q} (L;t ) \in \Z[t^{\pm 1}]$. Consider the logarithm $ \log ( P_0 (L,t))$ as a smooth function in $C^{\infty}(\R)$. We denote by $(\log ( P_0 (L,t))^{(n)} $ the $n$-th derivative of $ \log ( P_0 (L,t))$ evaluated at $t= 1$.

\begin{figure}[tpb]
\begin{center}
\begin{picture}(110,51)
\put(-96,27){\pc{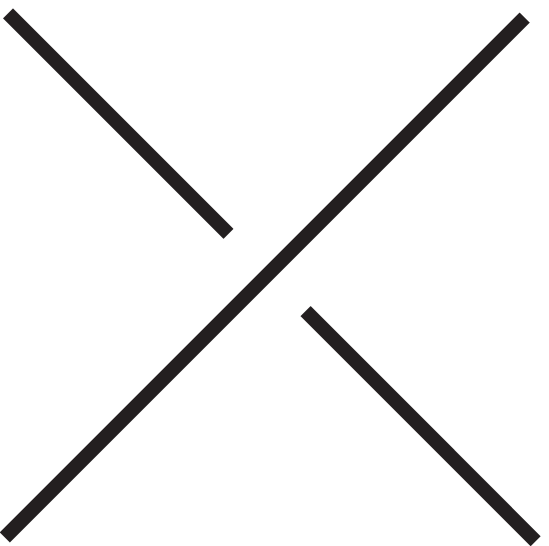}{0.300}}
\put(16,27){\pc{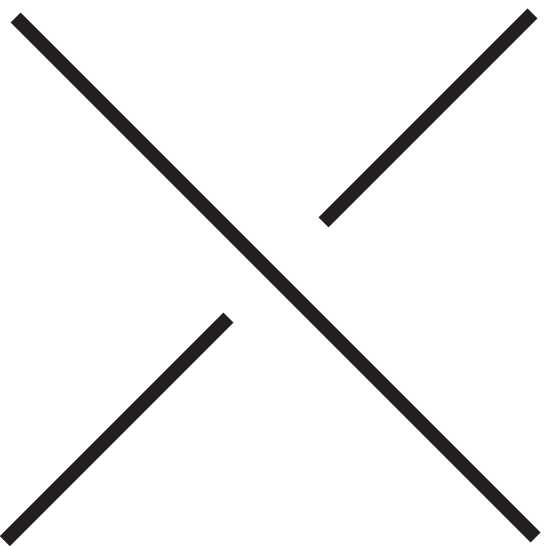}{0.300}}
\put(126,27){\pc{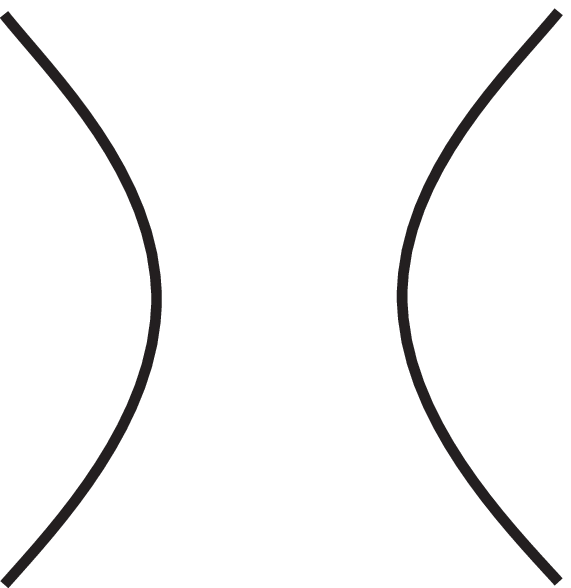}{0.300}}
\put(-116,43){\Large $L_+$}
\put(-6,43){\Large $L_-$}
\put(106,43){\Large $L_0$}
\end{picture}
\end{center}
\vskip -1.7pc
\caption{\label{t321} The links $L_+$, $ L_-$, and $L_0$.}
\end{figure}

Next, we review the notation used in \cite{MY}. Let $T = \cup_{i=1}^q T_i$ be a $q$-component string link in the 3-cube, and take an index $I= i_1 i_ 2 \cdots i_m \in \{ 1, \dots, q\}^m $. Let $r_i$ be the cardinality $\# \{ k | i_{k}=i \}$. We can define another link $D_I (T)$ of $q$-components in the following manner:
\begin{enumerate}[(1)]
\item Replace each string $T_i$ by $r_i$ zero-framed parallel copies of $T_i$. Here, we write $T_{(i,j)}$ for the $j$-th copy. If $r_i=0$ for some index $i$, we delete $L_i$.
\item We define $D_I(T)$ to be the $m$-component link $\cup_{i,j} T_{(i,j)}$ with the order induced by the lexicographic order of the index $(i,j)$. This ordering defines a bijection $\{ (i,j) | \ 1 \leq i \leq q, \ 1 \leq j \leq r_i\} \ra \{ 1, \dots, m\} $.
\end{enumerate}
In addition, we define a sequence $D_I (T) \in \{ 1, \dots, m\}^m $ without a repeated index as follows. First, we take a sequence of elements of $\{ (i,j) | \ 1 \leq i \leq q, \ 1 \leq j \leq r_i\} \ra \{ 1, \dots, m\} $ by replacing each $i$ in $I$ with $(i,1), \dots, (i,r_i )$ in this order. Next, we replace each term $(i,j)$ of this sequence with $\varphi((i,j)). $

In addition, given a subsequence $H < D(I)$, we define another link $D_I(T)_J$. Let $B_I$ be an oriented $2m$-gon, and denote by $p_j$ $(j=1, \dots, m)$ a set of $m$ nonadjacent edges of $B_I$ according to the boundary orientation. Suppose that $B_I$ is embedded in $S^3$ such that $B_I \cap L = \cup_{j=1}^m p_j$ and such that each $p_j$ is contained in $L_{i_j}$ with opposite orientation. We call such a disk {\it an $I$-fusion disk} of $D_I(L)$. For any subsequence $J$ of $D(I)$, we define the oriented link $L_J$ as the closure of
$$\bigcup_{j \in \{ J \} } (L_j \cap \partial B_I) \setminus \bigl( \bigcup_{j \in \{ J \} } (L_j \cap B_I) \bigr), $$
where $\{ J\}$ is the subset of $\{ 1, \dots, n\}$ of all indices appearing in the sequence $J $.

Consider the homomorphism $\mathcal{M}_m: F \ra \Z[X_1, \dots, X_q]/ {\mathcal{I}_m}$ defined by $ \mathcal{M}_m (x_i) =1+X_i $. For a sequence in $I=i_1 \cdots i_m \in \{ 1,2, \dots, q\}^m$, $\mu_{I}(T)$ is defined by the coefficient of $X_{i_1 } \cdots X_{i_{m-1}}$ in $\mathcal{M}_m(\lambda_{i_m})$ as in \cite{Mil2,Le,Orr,MY}.
\begin{thm}
[\cite{MY}]\label{ooth2} Let $T$ be a $q$-component string link which satisfies $\mathfrak{A}_{k+1} $. Assume $3 \leq m \leq 2k+2 $. Let $I$ be a sequence in $\{ 1,2, \dots, q\}^m$ of length $m$. For any $D_I$-fusion disk for $D_I (T)$, we have
$$ \mu_I (T ) = \frac{(-1)}{m ! 2^m }\sum_{J < D(I) } (-1)^{|J|} \log P_0(\overline{D_I( T )_J }) ^{(m)}. $$
\end{thm}
The original statement \cite[Theorem 1.3]{MY} dealt with only links $S^3$ and takes the same formula modulo some integers. However, as can be seen in their proof, after the authors proved Theorem \ref{ooth2} for string links before they proved the original statement. Thus, we do not need to give a detailed proof of Theorem \ref{ooth2}.

\subsection*{Acknowledgments}
The author expresses his gratitude to Professor Akira Yasuhara for his helpful discussion and for referring him to the paper \cite{Le}. He also thanks Yusuke Kuno, Jean-Baptiste Meilhan, and Kent Orr for giving him valuable comments on this work.

\vskip 1pc

\normalsize

DEPARTMENT OF
MATHEMATICS
TOKYO
INSTITUTE OF
TECHNOLOGY
2-12-1
OOKAYAMA
, MEGURO-KU TOKYO
152-8551 JAPAN

\end{document}